\documentclass[reqno]{amsart}

\date{March 11, 2016}

\usepackage{hyperref}
\usepackage{xspace}
\usepackage[utf8]{inputenc}
\usepackage{graphicx}
\usepackage{pstricks}
\usepackage{amssymb}
\usepackage{mathrsfs}
\usepackage[active]{srcltx}
\usepackage{url}
\usepackage{amsmath,amstext,amsxtra,amsgen,amsbsy,amsopn,amscd,amsthm,amsfonts}
\usepackage{latexsym}
\usepackage{dsfont}
\usepackage{enumitem}%pouvoir enlever l'indentation des itemize

\newtheorem{theorem}{Theorem}

\newtheorem{proposition}{Proposition}
\newtheorem{lemma}[proposition]{Lemma}

\theoremstyle{definition}

\theoremstyle{remark}

\newtheorem{remark}[proposition]{Remark}

\newcommand\R{{\ensuremath {\mathbb R} }}
\newcommand\C{{\ensuremath {\mathbb C} }}
\newcommand\N{{\ensuremath {\mathbb N} }}
\newcommand\Z{{\ensuremath {\mathbb Z} }}

\renewcommand\phi{\varphi}

\renewcommand\le{\leqslant}
\renewcommand\ge{\geqslant}

\renewcommand\epsilon{\varepsilon}
\renewcommand\hat{\widehat}

\renewcommand\bar{\overline}

\newcommand{\gH}{\mathfrak{H}}

\newcommand{\gS}{\mathfrak{S}}

\newcommand\ii{{\ensuremath {\infty}}}

\newcommand{\norm}[1]{ \left| \! \left| #1 \right| \! \right| }

\newcommand{\cF}{\mathcal{F}}

\newcommand{\cR}{\mathcal{R}}

\newcommand\1{{\ensuremath {\mathds 1} }}

\DeclareMathOperator{\tr}{Tr}

\title{The Stein-Tomas inequality in trace ideals}

\author{Rupert L. Frank}
\address{Rupert L. Frank, Mathematics 253-37, Caltech, Pasadena, CA 91125, USA}
\email{rlfrank@caltech.edu}

\author[J. Sabin]{Julien SABIN}
\address{J. Sabin, Laboratoire de Mathématiques d'Orsay, Univ. Paris-Sud, CNRS, Université
Paris-Saclay, 91405 Orsay, France.} 
\email{julien.sabin@math.u-psud.fr}

\date{}

\begin{document}
 
\begin{abstract}
 The goal of this review is to explain some recent results \cite{FraSab-14} regarding generalizations of the Stein-Tomas (and Strichartz) inequalities to the context of trace ideals (Schatten spaces).
\end{abstract}

\maketitle

\section*{Introduction}

As a general question in Fourier analysis, one wants to understand how the Fourier transform acts on $L^p(\R^N)$, when $1<p<2$ and $N\ge1$. When $f\in L^p(\R^N)$ with $p=1$, the Fourier transform $\hat{f}$ of $f$ is a continuous function on $\R^N$ vanishing at infinity, while for $p=2$ it is merely a square integrable function by Plancherel's theorem. In particular, $\hat{f}$ may be restricted to any given set of zero Lebesgue measure in $\R^N$ in a meaningful way when $p=1$, a property which is not true anymore when $p=2$ since $\hat{f}$ is only defined almost everywhere in this case. It is then a natural question to ask whether this property persists for some $p>1$, even though the Hausdorff-Young inequality implies that $\hat{f}$ belongs to $L^{p'}(\R^N)$ in this case (where we denoted by $p'$ the dual exponent of $p$) and thus is also only defined almost everywhere a priori. A positive answer to this question for some particular zero measure sets has been provided by the striking result of Stein and Tomas.

\begin{theorem}[Stein \cite{Stein-86}-Tomas \cite{Tomas-75}]\label{thm:stein-tomas}
  Let $N\ge2$ and $S\subset\R^N$ a compact hypersurface with non-vanishing Gauss curvature. Then, for any $1\le p\le\frac{2(N+1)}{N+3}$, there exists $C>0$ such that for all $f\in (L^1\cap L^p)(\R^N)$, we have
  \begin{equation}\label{eq:stein-tomas}
      \norm{\hat{f}_{\,|S}}_{L^2(S)}\le C\norm{f}_{L^p(\R^N)},
  \end{equation}
  where $S$ is endowed with its Lebesgue measure. Furthermore, the exponent $p=\frac{2(N+1)}{N+3}$ is optimal in the sense that \eqref{eq:stein-tomas} is wrong for any $p>\frac{2(N+1)}{N+3}$. 
\end{theorem}

Theorem \ref{thm:stein-tomas} provides a distinction between the Fourier transform on $L^p(\R^N)$ whether $p<2(N+1)/(N+3)$ or not: when it is the case, the range of the Fourier transform on $L^p(\R^N)$ lies within a subset of $L^{p'}(\R^N)$ in which functions may be restricted to any curved compact surface in a $L^2$-sense,  a property which does not hold when $p>2(N+1)/(N+3)$. 

The restriction property of Stein and Tomas is stated in terms of square integrable functions on the surface $S$. One may wonder whether the exponent $p$ can be increased if we replace $L^2(S)$ by the larger space $L^1(S)$ in \eqref{eq:stein-tomas}: this is the content of the famous Stein-Tomas conjecture, which states that one can go up to $p<2N/(N+1)$ in this case. For now, this conjecture has been proved only in $N=2$ when $S$ is a circle and is open in any other case. One may consult the review article \cite{Tao-04} for informations on this conjecture.

It is perhaps surprising that curvature plays a role in these properties of the Fourier transform. However, examining the case when $S$ is flat, for instance if $S$ is a portion of the hyperplane $\{\xi\in\R^N,\, \xi_1=0\}$, one sees that the function
$$f(x)=\frac{1}{1+|x_1|}\chi(x_2,\ldots,x_N),\,x\in\R^N,$$
with $\chi\in C^\ii_0(\R^N)$, belongs to $L^p(\R^N)$ for any $p>1$ (but not $p=1$), and has a Fourier transform which blows up on $S$; hence the property \eqref{eq:stein-tomas} can only hold for $p=1$ for a flat surface (to rigorously negate the inequality \eqref{eq:stein-tomas}, one has to consider a well-chosen sequence of integrable functions approaching this $f$ in $L^p(\R^N)$). 

The reason why curvature indeed implies restriction properties of the Fourier transform for $p>1$ is actually related to the choice of the space $L^2(S)$ in \eqref{eq:stein-tomas}. Denoting by $\cR_S$ the restriction operator $\cR_Sf=\hat{f}_{|S}$, Theorem \ref{thm:stein-tomas} says that the linear operator $\cR_S$ extends to a bounded operator from $L^p(\R^N)$ to $L^2(S)$. Using the Hilbert space structure of $L^2(S)$, this is equivalent to the boundedness of the translation-invariant operator $T_S:=(\cR_S)^*\cR_S$ from $L^p(\R^N)$ to $L^{p'}(\R^N)$. A straightforward computation shows that $T_S f=f*\hat{d\sigma}$, where
$$\hat{d\sigma}(x):=\int_S e^{ix\cdot\xi}d\sigma(\xi),\,\forall x\in\R^N,$$
$d\sigma$ denoting the Lebesgue measure on $S$. The curvature of $S$ then implies by a stationary phase argument the decay estimate
$$|\hat{d\sigma}(x)|\le C(1+|x|)^{-\frac{N-1}{2}},$$
which, using the Hardy-Littlewood-Sobolev inequality, leads to Theorem \ref{thm:stein-tomas} in the range 
$1\le p\le 4N/(3N+1)$, a range strictly smaller than the optimal one $1\le p\le 2(N+1)/(N+3)$. It is due to the fact that $\hat{d\sigma}$ is not any function which decays like $|x|^{-(N-1)/2}$ at infinity: one also has to use that its Fourier transform is supported on a $(N-1)$-dimensional object. Using this additional information, Tomas \cite{Tomas-75} was able to obtain Theorem \ref{thm:stein-tomas} in the subcritical range $1\le p<2(N+1)/(N+3)$, and then Stein \cite{Stein-86} obtained the full optimal range using a complex interpolation argument (which we will explain in more details in Section \ref{sec:proof}).

An interesting extension to the Stein-Tomas theorem has been provided by Strichartz \cite{Strichartz-77}, who considered non-compact surfaces which are levels of quadratic forms. Of particular interest is the case of the paraboloid
$$S=\{(\omega,\xi)\in\R\times\R^d,\,\omega=-|\xi|^2\},$$
for some $d\ge1$, for which Strichartz obtained Theorem \ref{thm:stein-tomas} when $p=2(d+2)/(d+4)$, corresponding to the endpoint exponent $p=2(N+1)/(N+3)$ since $N=d+1$ in this case. Here, the measure on $S$ used to define the space $L^2(S)$ is obtained by pushing forward the Lebesgue measure on $\R^d$ by the map $\xi\in\R^d\mapsto(\xi,-|\xi|^2)\in S$, inducing an isometry between $L^2(S)$ and $L^2(\R^d)$. Using this identification, one may verify that for any $g\in L^2(\R^d)\simeq L^2(S)$ we have
$$(\cR_S)^*g(t,x)=(e^{it\Delta_x}\hat{g})(x),\,\forall (t,x)\in\R\times\R^d,$$
where $\Delta_x$ denotes the Laplace operator on $\R^d$. Hence, Strichartz' result implies the estimate 
$$\norm{e^{it\Delta_x}g}_{L^{2+4/d}_{t,x}(\R\times\R^d)}\le C\norm{g}_{L^2(\R^d)},\,\forall g\in L^2(\R^d),$$
which is of particular interest since the function $g(t,x)=e^{it\Delta_x}g$ is the unique solution to the Schr\"odinger equation $i\partial_t g=-\Delta_x g$ with initial data $g(0,\cdot)=g$. One may thus see Strichartz' estimate as a way to measure the smallness of $g(t,\cdot)$ for large times $t$, which is an expression of the dispersive properties of the linear Schr\"odinger equation. Strichartz obtained similar estimates when $S$ is a cone or a two-sheeted hyperboloid, which correspond respectively to the wave equation and to the Klein-Gordon equation. These kinds of dispersive estimates are a decisive ingredient in the study of non-linear versions of these equations. 

We have seen that Stein-Tomas-type theorems amount to proving the boundedness of the linear operator $T_S$ from $L^p(\R^N)$ to $L^{p'}(\R^N)$. By the H\"older inequality, this is equivalent to the boundedness of the operator $W_1T_SW_2$ on $L^2(\R^N)$ (where $W_1$ and $W_2$ are seen as multiplication operators), for any $W_1,W_2\in L^{2p/(2-p)}(\R^N)$ together with the bound
\begin{equation}\label{eq:stein-tomas-hilbert}
 \norm{W_1T_SW_2}_{L^2\to L^2}\le C\norm{W_1}_{L^{2p/(2-p)}(\R^N)}\norm{W_2}_{L^{2p/(2-p)}(\R^N)},
\end{equation}
for some $C>0$ independent of $W_1$ and $W_2$. In \cite{FraSab-14}, we obtain additional information on the operator $W_1T_SW_2$ beyond its mere boundedness: we show that it is actually a compact operator and that it belongs to certain Schatten classes (recall that a compact operator $A$ on a Hilbert space $\gH$ belongs to the Schatten class $\gS^\alpha$ (with $\alpha\ge1$) if the sequence of its singular values $(\mu_n(A))$ belongs to the sequence space $\ell^\alpha$, which induces a norm on $\gS^\alpha$ by $\norm{A}_{\gS^\alpha}=\norm{(\mu_n(A))}_{\ell^\alpha}$ \cite{Simon-79}). In other words, we upgrade the Stein-Tomas inequality \eqref{eq:stein-tomas-hilbert} to the stronger
\begin{equation}\label{eq:stein-tomas-schatten}
 \norm{W_1T_SW_2}_{\gS^\alpha}\le C\norm{W_1}_{L^{2p/(2-p)}(\R^N)}\norm{W_2}_{L^{2p/(2-p)}(\R^N)},
\end{equation}
with the optimal (that is, smallest; the Stein-Tomas inequality \eqref{eq:stein-tomas-hilbert} corresponding to $\alpha=\ii$) exponent $\alpha=\alpha(p)\ge1$ (see Theorem \ref{thm:main} for a precise statement). Since $T_S$ is a translation-invariant operator, it may be written as $T_S=f(-i\nabla)$ for some distribution $f$ (which here is supported on $S$). Inequalities of the form \eqref{eq:stein-tomas-schatten} are known to hold when $f$ belongs to some Lebesgue space, by results of Kato-Seiler-Simon \cite[Ch. 4]{Simon-79}. As a consequence, the inequality \eqref{eq:stein-tomas-schatten} may be seen as an extension of these results in the case when the distribution $f$ is singular (here, supported on a curved hypersurface). 

A motivation to look at trace ideals extensions to Stein-Tomas-type inequalities comes from a dual version of the inequality \eqref{eq:stein-tomas-schatten}: it is equivalent (see \cite[Lem. 3]{FraSab-14}) to the fact that for all orthonormal system $(f_j)$ in $L^2(S)$ and for any corresponding coefficients $(\nu_j)\subset\C$, one has the estimate
\begin{equation}\label{eq:stein-tomas-orthonormal}
  \norm{\sum_j\nu_j|(\cR_S)^*f_j|^2}_{L^{p'/2}(\R^N)}\le C\left(\sum_j|\nu_j|^{\alpha'}\right)^{1/\alpha'},
\end{equation}
for some $C>0$ independent of $(f_j)$ and $(\nu_j)$. Here, the exponent $\alpha'$ is dual to the exponent $\alpha$ appearing in \eqref{eq:stein-tomas-schatten}. Inequality \eqref{eq:stein-tomas-orthonormal} reduces to (the dual version of) \eqref{eq:stein-tomas} when the orthonormal system is reduced to one function (and the corresponding coefficient $\nu=1$). On the other hand, one cannot deduce \eqref{eq:stein-tomas-orthonormal} ``directly'' from \eqref{eq:stein-tomas} using the triangle inequality: this only leads to
$$\norm{\sum_j\nu_j|(\cR_S)^*f_j|^2}_{L^{p'/2}(\R^N)}\le\sum_j|\nu_j|\norm{(\cR_S)^*f_j}_{L^{p'}(\R^N)}^2\le C\sum_j|\nu_j|,$$
which is weaker than \eqref{eq:stein-tomas-orthonormal} if $\alpha'>1$. Hence, \eqref{eq:stein-tomas-orthonormal} may be seen as a generalization of \eqref{eq:stein-tomas} for systems of orthonormal functions, with an \emph{optimal} dependence on the number of such functions (that is, with an optimal exponent $\alpha$). Quantities appearing on the left-side of \eqref{eq:stein-tomas-orthonormal} arise in the context of many-body quantum mechanics: for instance, when $S$ is a paraboloid, the quantity
$$\sum_{j=1}^M|e^{it\Delta}u_j|^2$$
represents the spatial density at time $t$ of a system of $M$ fermions evolving freely in $\R^d$, the $j^{th}$ fermion having $u_j$ as wavefunction at time $t=0$ (the $M$-body wavefunction is then a \emph{Slater determinant} $u_1\wedge\cdots\wedge u_M$) . It is thus interesting to control such quantities for large times, but also for large $M$. We refer to \cite{LewSab-13a,LewSab-13b,Sabin-2014} for applications of such inequalities in the study of nonlinear PDEs modelling the evolution of infinite quantum systems. The inequality \eqref{eq:stein-tomas-orthonormal} can also be stated in a more concise way using the language of one-body density matrices \cite{FraSab-14}.

In the rest of the review, we state more precisely our results, prove some improvements which are new, and give an application to a refined Strichartz estimate. Finally, we provide elements of proof.

\section{Results}

\subsection{Summary}

We have generalized the results of Stein-Tomas in the compact case and the results of Strichartz in the non-compact case. Let us start with the compact case \cite[Thm. 2]{FraSab-14}.

\begin{theorem}\label{thm:main}
 Let $N\ge2$ and $S\subset\R^N$ a compact hypersurface with non-vanishing Gauss curvature. Then, for any $1\le p\le\frac{2(N+1)}{N+3}$ there exists $C>0$ such that for all $W_1,W_2\in L^{2p/(2-p)}(\R^N)$, we have 
 \begin{equation}\label{eq:stein-tomas-schatten-2}
    \norm{W_1T_SW_2}_{\gS^{\frac{(N-1)p}{2N-(N+1)p}}}\le C\norm{W_1}_{L^{2p/(2-p)}(\R^N)}\norm{W_1}_{L^{2p/(2-p)}(\R^N)}.
 \end{equation}
\end{theorem}

The Schatten exponent $\alpha=(N-1)p/(2N-(N+1)p)$ in \eqref{eq:stein-tomas-schatten-2} is optimal (that is, the smallest possible). It is proved in \cite[Thm. 6]{FraSab-14}, where the density matrix formalism is exploited: the orthonormal system used to test the inequality \eqref{eq:stein-tomas-orthonormal} corresponds to the (unknown) eigenfunctions of an operator $\gamma_h$ on $L^2(S)$ with integral kernel
$$\gamma_h(\xi,\xi')=\int_{\R^N}\1(|x|^2\le h^{-2})e^{ix\cdot(\xi-\xi')}\,dx,\,(\xi,\xi')\in S\times S,$$
for a small 'semi-classical' parameter $h>0$. As another remark, we notice that the Schatten exponent in \eqref{eq:stein-tomas-schatten-2} is equal to 1 (that is, the smallest possible) when $p=1$. While this looks like the stronger result, it is actually straightforward to prove it since in this case $W_1,W_2\in L^2(\R^N)$ and hence the operator $W_1(\cR_S)^*$ has an integral kernel 
$$W_1(\cR_S)^*(x,\xi)=W_1(x)e^{ix\cdot\xi},\,(x,\xi)\in\R^N\times S$$
which is square integrable on $\R^N\times S$, meaning that it is Hilbert-Schmidt (i.e. $\gS^2$) from $L^2(S)$ to $L^2(\R^N)$ and hence $W_1T_SW_2=W_1(\cR_S)^*\cR_SW_2$ is trace-class (i.e. $\gS^1$). 

When $S$ is a level set of a quadratic form (and hence can be non-compact), which is the case considered by Strichartz \cite{Strichartz-77}, we also obtain similar estimates \cite[Thm. 3]{FraSab-14}. For the sake of clarity, let us state it only in the case of a paraboloid, even if the result holds for different kind of quadratic surfaces. In the special case of the paraboloid, we can actually go a little bit further than what is proved in Strichartz' article by using mixed (space-time) Lebesgue norms \cite[Thm. 9]{FraSab-14}.

\begin{theorem}\label{thm:schatten-strichartz-mixed}
  Let $d\ge1$ and $S$ be the paraboloid
  $$S:=\{(\omega,\xi)\in\R\times\R^d,\,\omega=-|\xi|^2\}.$$
  Then, for all exponents $p,q\ge1$ satisfying the relations 
  $$\frac2p+\frac dq=1,\qquad q>d+1,$$
  there exists $C>0$ such that for all $W_1,W_2\in L^p_tL^q_x(\R\times\R^d)$ we have 
  \begin{equation}\label{eq:schatten-strichartz-mixed}
    \norm{W_1T_SW_2}_{\gS^q(L^2(\R^{d+1}))}\le C\norm{W_1}_{L^p_tL^q_x(\R\times\R^d)}\norm{W_2}_{L^p_tL^q_x(\R\times\R^d)}.
  \end{equation}
\end{theorem}

Contrary to the compact case of Theorem \ref{thm:main}, we see here that the Schatten class $\gS^1$ is never attained, which is due to the fact that $S$ has infinite volume in this case and hence the argument presented to obtain the trace-class property does not hold anymore. Theorem \ref{thm:schatten-strichartz-mixed} was proved for the first time by Frank, Lewin, Lieb, and Seiringer in \cite{FraLewLieSei-13} in the restricted range $q\ge d+2$. They also proved the following statement concerning the optimality of the Schatten exponent.

\begin{proposition}\label{prop:optimality}
  Let $d,p,q\ge1$ exponents satisfying $2/p+d/q=1$. Assume that there exists $\alpha\ge1$ and $C>0$ such that for all Schwartz function $W$ one has
  \begin{equation}\label{eq:schatten-strichartz-mixed-2}
    \norm{\bar{W}T_SW}_{\gS^\alpha(L^2(\R^{d+1}))}\le C\norm{W}_{L^p_tL^q_x(\R\times\R^d)}^2.
  \end{equation}
  Then, the Schatten exponent $\alpha$ must satisfy
  $$\alpha\ge q,\qquad \alpha>d+1.$$
\end{proposition}

This last statement implies that the range $q>d+1$ is optimal for \eqref{eq:schatten-strichartz-mixed-2} to hold with $\alpha=q$, as in Theorem \ref{thm:schatten-strichartz-mixed}. However, one can go below $q=d+1$ if the Schatten exponent is not $\alpha=q$. Indeed, in the numerology of Theorem \ref{thm:schatten-strichartz-mixed}, the Keel-Tao endpoint \cite{KeeTao-98} corresponds to $q=d$ (for $d\ge3$) and $q=2$ (if $d=1$), while for $d=2$ the endpoint Strichartz estimate is known to fail. As a consequence, one cannot hope to have \eqref{eq:schatten-strichartz-mixed-2} for $q<d$ (while the relation $2/p+d/q=1$ is imposed by scaling). The result of Keel-Tao implies that we have the estimate \eqref{eq:schatten-strichartz-mixed-2} for $q=d$ ($d\ge3)$ or $q=2$ ($d=1$) with the Schatten exponent $\alpha=\ii$ (which corresponds to the operator norm). Interpolating between the cases $q=d$ ($\alpha=\ii$) and $q>d+1$ ($\alpha=q$), we obtain the estimate \eqref{eq:schatten-strichartz-mixed-2} in the range $d\le q\le d+1$ with $\alpha=q/(q-d)+\epsilon$ for all $\epsilon>0$. The results are depicted in Figure~\ref{figure:strichartz-1}, where the purple region denotes the set of exponents for which we know that \eqref{eq:schatten-strichartz-mixed} holds. Away from the union of the purple and dashed regions, the estimate \eqref{eq:schatten-strichartz-mixed-2} fails due to Proposition \ref{prop:optimality}. Hence, the question remains whether \eqref{eq:schatten-strichartz-mixed-2} holds in the dashed region. We discuss these (new) results in the next section.

\begin{center}
\begin{figure}
 \scalebox{.75}{
 \ifx\JPicScale\undefined\def\JPicScale{1}\fi
  \psset{unit=\JPicScale mm}
  \newrgbcolor{userFillColour}{0 0 0}
  \psset{linewidth=0.3,dotsep=1,hatchwidth=0.3,hatchsep=1.5,shadowsize=1,dimen=middle}
  \psset{dotsize=0.7 2.5,dotscale=1 1,fillcolor=userFillColour}
  \psset{arrowsize=1 2,arrowlength=1,arrowinset=0.25,tbarsize=0.7 5,bracketlength=0.15,rbracketlength=0.15}
  \makeatletter
  %\@ifundefined{Pst@correctAnglefalse}{}{\psset{correctAngle=false}}
  \makeatother
    \begin{pspicture}(4,3)(152,67)
    \psline[linewidth=0.5](110,9)(110,7)
    \pscustom[linestyle=none,fillstyle=vlines]{\psline(140,9)(140,60)
    \psline(140,60)(110,60)
    \psline(110,60)(140,9)
    \psbezier(140,9)(140,9)(140,9)
    }
    \psline[linewidth=0.45]{<-}(10,67)(10,9)
    \newrgbcolor{userFillColour}{0.6 0.6 1}
    \pscustom[linestyle=none,fillcolor=userFillColour,fillstyle=solid]{\psline(10,9)(110,60)
    \psline(110,60)(140,9)
    \psline(140,9)(10,9)
    \psbezier(10,9)(10,9)(10,9)
    }
    \newrgbcolor{userLineColour}{0.6 0.6 1}
    \psline[linewidth=0.5,linecolor=userLineColour](10,9)(110,60)
    \psline[linewidth=0.5](140,9)(140,7)
    \psline[linewidth=0.5](8,60)(10,60)
    \rput(4,60){$\frac{1}{d+1}$}
    \rput(110,3){$\frac{1}{d+1}$}
    \rput(140,3){$\frac{1}{d}$}
    \rput(152,9){$\frac{1}{q}$}
    \rput(4,67){$\frac{1}{\alpha}$}
    \newrgbcolor{userLineColour}{0.6 0.6 1}
    \psline[linecolor=userLineColour](10,9)(110,60)
    \newrgbcolor{userFillColour}{1 1 1}
    \rput{0}(110,60){\psellipse[fillcolor=userFillColour,fillstyle=solid](0,0)(1,-1)}
    \psline{->}(56,46)(62,37)
    \rput(54,50){$\frac{1}{\alpha}=\frac{1}{q}$}
    \psline[linewidth=0.45]{->}(10,9)(147,9)
    \newrgbcolor{userLineColour}{0.6 0.6 1}
    \psline[linecolor=userLineColour](10,9)(140,9)
    \newrgbcolor{userFillColour}{0.6 0.6 1}
    \rput{0}(140,9){\psellipse[fillcolor=userFillColour,fillstyle=solid](0,0)(1,-1)}
    \rput(8,7){$0$}
\end{pspicture}
 }
 \caption{Range of Strichartz inequality ($d\ge3$). Purple: valid. Dashed region: unknown by previous works.\label{figure:strichartz-1}} 

\end{figure}
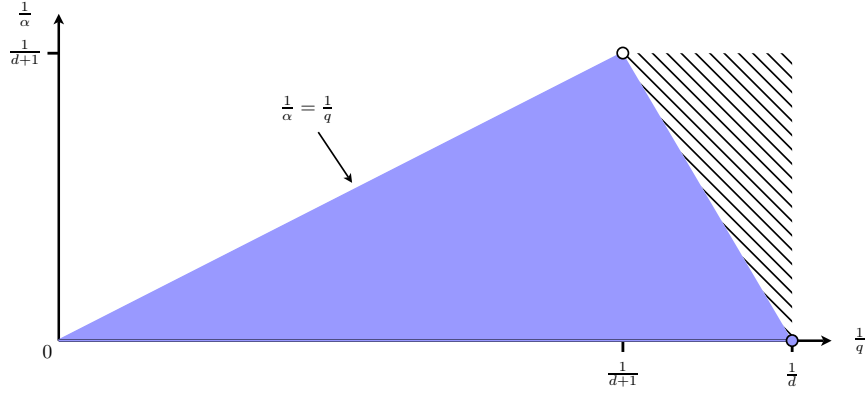
\end{center}

\subsection{New results}

In this section, we prove the following two new results:

\begin{lemma}\label{lem:optimality-keel-tao}
 Let $d\ge3$ and $S$ the paraboloid in $\R^{d+1}$. Then, for any non-zero $W\in L^d_x(\R^d)$, the bounded operator $\bar{W}T_SW$ is not compact in $L^2(\R^{d+1})$. 
\end{lemma}

\begin{proposition}\label{prop:new-necessary-condition}
 If the inequality \eqref{eq:schatten-strichartz-mixed-2} holds, then $\alpha\ge q/(q-d)$. 
\end{proposition}

\begin{remark}
 Lemma \ref{lem:optimality-keel-tao} means that $\alpha=\ii$ is optimal in the estimate \eqref{eq:schatten-strichartz-mixed-2} at the Keel-Tao endpoint $q=d$. Indeed, notice that when $q=d$, the exponent $p$ in the statement of Theorem \ref{thm:schatten-strichartz-mixed} is equal to $p=\ii$. Hence, in Lemma \ref{lem:optimality-keel-tao} the function $W=W(x)\in L^d_x(\R^d)$ is identified to the (constant in time) function $W=W(t,x)=W(x)\in L^\ii_tL^d_x(\R\times\R^d)$. In other words, this means that one cannot do better than the triangle inequality at the Keel-Tao endpoint. 
\end{remark}

\begin{remark}
 Proposition \ref{prop:new-necessary-condition} means that the inequality \eqref{eq:schatten-strichartz-mixed-2} actually \emph{fails} in the dashed region of Figure \ref{figure:strichartz-1}.
\end{remark}

According to Proposition \ref{prop:new-necessary-condition}, the only place where the validity of \eqref{eq:schatten-strichartz-mixed-2} is not known is the dashed line represented on Figure \ref{fig:main}. It would follow from proving an endpoint estimate at $q=d+1$. At this point, it is conjectured that \eqref{eq:schatten-strichartz-mixed-2} holds in the weak Schatten space $\gS^{d+1}_{\text{w}}$ rather than in $\gS^{d+1}$ (for which the estimate does not hold). However, this weak Schatten estimate is completely open.

\begin{center}
\begin{figure}
 \scalebox{.75}{
 
 \ifx\JPicScale\undefined\def\JPicScale{1}\fi
\psset{unit=\JPicScale mm}
\newrgbcolor{userFillColour}{0 0 0}
\psset{linewidth=0.3,dotsep=1,hatchwidth=0.3,hatchsep=1.5,shadowsize=1,dimen=middle}
\psset{dotsize=0.7 2.5,dotscale=1 1,fillcolor=userFillColour}
\psset{arrowsize=1 2,arrowlength=1,arrowinset=0.25,tbarsize=0.7 5,bracketlength=0.15,rbracketlength=0.15}
\makeatletter
%\@ifundefined{Pst@correctAnglefalse}{}{\psset{correctAngle=false}}
\makeatother
\begin{pspicture}(4,3)(152,67)
\psline[linewidth=0.45]{<-}(10,67)(10,9)
\newrgbcolor{userFillColour}{0.6 0.6 1}
\pscustom[linestyle=none,fillcolor=userFillColour,fillstyle=solid]{\psline(10,9)(110,60)
\psline(110,60)(140,9)
\psline(140,9)(10,9)
\psbezier(10,9)(10,9)(10,9)
}
\psline[linewidth=0.5,linestyle=dotted](110,60)(140,9)
\newrgbcolor{userLineColour}{0.6 0.6 1}
\psline[linewidth=0.5,linecolor=userLineColour](10,9)(110,60)
\psline[linewidth=0.5](110,11)(110,7)
\psline[linewidth=0.5](140,11)(140,7)
\psline[linewidth=0.5](8,60)(13,60)
\rput(4,60){$\frac{1}{d+1}$}
\rput(110,3){$\frac{1}{d+1}$}
\rput(140,3){$\frac{1}{d}$}
\rput(152,9){$\frac{1}{q}$}
\rput(4,67){$\frac{1}{\alpha}$}
\psline(10,9)(110,60)
\newrgbcolor{userFillColour}{1 1 1}
\rput{0}(110,60){\psellipse[fillcolor=userFillColour,fillstyle=solid](0,0)(1,-1)}
\psline{->}(56,46)(62,37)
\rput(54,50){$\frac{1}{\alpha}=\frac{1}{q}$}
\psline{<-}(125,38)(135,44)
\rput(138,48){$\frac{1}{\alpha}=1-\frac{d}{q}$}
\psline[linewidth=0.45]{->}(10,9)(147,9)
\end{pspicture}

 }
 \caption{Range of optimal Strichartz inequality (purple). Dashed line: conjectured to hold.\label{fig:main}} 
\end{figure}
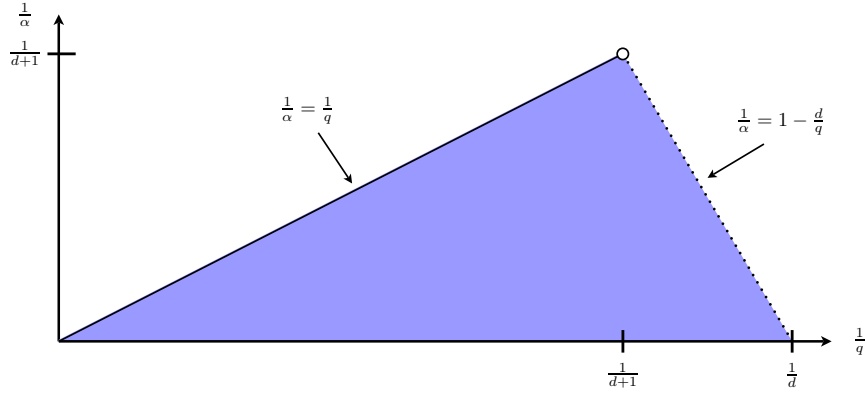
\end{center}

We now turn to the proofs of Lemma \ref{lem:optimality-keel-tao} and Proposition \ref{prop:new-necessary-condition}.

\begin{proof}[Proof of Lemma \ref{lem:optimality-keel-tao}]
 Recall that $\bar{W}T_SW=\bar{W}(\cR_S)^*\cR_SW=(\cR_S W)^*\cR_S W$, so that the compactness of $\bar{W}T_SW$ is equivalent to the compactness of 
 $$\cF_x\cR_S W(\cR_S W)^*\cF_x=\cF_x\cR_S|W|^2(\cR_S)^*\cF_x=\int_\R e^{-it\Delta_x}|W(t,x)|^2 e^{it\Delta_x}\,dt$$
 on $L^2_x(\R^d)$, where $\cF_x$ denotes the Fourier transform on $L^2_x(\R^d)$. Using that $W(t,x)=W(x)$ is independent of $t$, we find that for $\phi_n:=e^{in\Delta_x}\phi$, we have by a simple change of variables
 \begin{align*}
    \langle\phi_n,\int_\R e^{-it\Delta_x}|W(t,x)|^2 e^{it\Delta_x}\,dt\,\phi_n\rangle &= \langle\phi,\int_\R e^{-i(t+n)\Delta_x}|W(x)|^2 e^{i(t+n)\Delta_x}\,dt\,\phi\rangle \\
    &= \langle\phi,\int_\R e^{-it\Delta_x}|W(x)|^2 e^{it\Delta_x}\,dt\,\phi\rangle \\
    &= \int_\R\norm{We^{it\Delta_x}\phi}_{L^2_x}^2\,dt,
 \end{align*}
 which is independent of $n$ and certainly non-zero since $W\neq0$, for an adequate choice of $\phi\in L^2_x(\R^d)$. Since $(\phi_n)$ goes weakly to zero in $L^2_x(\R^d)$ this shows that the operator 
 $$\int_\R e^{-it\Delta_x}|W(t,x)|^2 e^{it\Delta_x}\,dt$$
 is not compact on $L^2_x(\R^d)$.
\end{proof}

When $d=1$, we have $q=2$ and hence $p=4$. Hence, the previous argument (which relied on the fact that $p=\ii$) fails in this case and in principle, we may expect that one can do better than $\alpha=\ii$ at the endpoint $q=2$ when $d=1$.

\begin{proof}[Proof of Proposition \ref{prop:new-necessary-condition}]
 We claim that if \eqref{eq:schatten-strichartz-mixed-2} holds for $\alpha=d+2$, then $q\ge d(d+2)/(d+1)$. Graphically, this means that we show that the inequality fails on an open half-line as depicted on Figure \ref{fig:argument-1}.

\begin{center}
 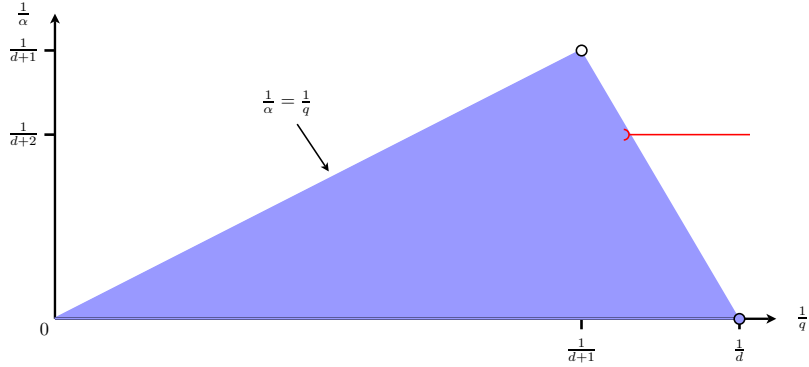
\begin{figure}
 \scalebox{.7}{
 
 \ifx\JPicScale\undefined\def\JPicScale{1}\fi
\psset{unit=\JPicScale mm}
\newrgbcolor{userFillColour}{0 0 0}
\psset{linewidth=0.3,dotsep=1,hatchwidth=0.3,hatchsep=1.5,shadowsize=1,dimen=middle}
\psset{dotsize=0.7 2.5,dotscale=1 1,fillcolor=userFillColour}
\psset{arrowsize=1 2,arrowlength=1,arrowinset=0.25,tbarsize=0.7 5,bracketlength=0.15,rbracketlength=0.15}
\makeatletter
%\@ifundefined{Pst@correctAnglefalse}{}{\psset{correctAngle=false}}
\makeatother
\begin{pspicture}(4,3)(152,67)
\psline[linewidth=0.5](110,9)(110,7)
\psline[linewidth=0.45]{<-}(10,67)(10,9)
\newrgbcolor{userFillColour}{0.6 0.6 1}
\pscustom[linestyle=none,fillcolor=userFillColour,fillstyle=solid]{\psline(10,9)(110,60)
\psline(110,60)(140,9)
\psline(140,9)(10,9)
\psbezier(10,9)(10,9)(10,9)
}
\newrgbcolor{userLineColour}{0.6 0.6 1}
\psline[linewidth=0.5,linecolor=userLineColour](10,9)(110,60)
\psline[linewidth=0.5](140,9)(140,7)
\psline[linewidth=0.5](8,60)(10,60)
\rput(4,60){$\frac{1}{d+1}$}
\rput(110,3){$\frac{1}{d+1}$}
\rput(140,3){$\frac{1}{d}$}
\rput(152,9){$\frac{1}{q}$}
\rput(4,67){$\frac{1}{\alpha}$}
\newrgbcolor{userLineColour}{0.6 0.6 1}
\psline[linecolor=userLineColour](10,9)(110,60)
\newrgbcolor{userFillColour}{1 1 1}
\rput{0}(110,60){\psellipse[fillcolor=userFillColour,fillstyle=solid](0,0)(1,-1)}
\psline{->}(56,46)(62,37)
\rput(54,50){$\frac{1}{\alpha}=\frac{1}{q}$}
\psline[linewidth=0.45]{->}(10,9)(147,9)
\newrgbcolor{userLineColour}{0.6 0.6 1}
\psline[linecolor=userLineColour](10,9)(140,9)
\newrgbcolor{userFillColour}{0.6 0.6 1}
\rput{0}(140,9){\psellipse[fillcolor=userFillColour,fillstyle=solid](0,0)(1,-1)}
\rput(8,7){$0$}
\psline[linewidth=0.5](8,44)(10,44)
\rput(4,44){$\frac{1}{d+2}$}
\rput{0}(118,44){\psellipticarc[linecolor=red](0,0)(1,-1){-90}{90}}
\psline[linecolor=red](142,44)(119,44)
\end{pspicture}

 }
 \caption{Open half-line (red) of invalidity of \eqref{eq:schatten-strichartz-mixed-2}.\label{fig:argument-1}} 
\end{figure}
\end{center}

 Once this is shown, this implies the proposition. Indeed, imagine that the inequality holds for a point on the previously dashed region. Then, by complex interpolation, it must hold on the convex hull of the purple region and this additional point. This convex hull clearly intersects the red open half-line, as depicted in the Figure \ref{fig:argument-2}, which leads to a contradiction. 
 
\begin{center}
  \begin{figure}
  \scalebox{.7}{
  
  \ifx\JPicScale\undefined\def\JPicScale{1}\fi
\psset{unit=\JPicScale mm}
\newrgbcolor{userFillColour}{0 0 0}
\psset{linewidth=0.3,dotsep=1,hatchwidth=0.3,hatchsep=1.5,shadowsize=1,dimen=middle}
\psset{dotsize=0.7 2.5,dotscale=1 1,fillcolor=userFillColour}
\psset{arrowsize=1 2,arrowlength=1,arrowinset=0.25,tbarsize=0.7 5,bracketlength=0.15,rbracketlength=0.15}
\makeatletter
%\@ifundefined{Pst@correctAnglefalse}{}{\psset{correctAngle=false}}
\makeatother
\begin{pspicture}(4,3)(152,67)
\pscustom[linestyle=none,fillstyle=hlines]{\psline(136,38)(110,60)
\psline(110,60)(140,9)
\psline(140,9)(136,38)
\psbezier(136,38)(136,38)(136,38)
}
\psline[linewidth=0.5](110,9)(110,7)
\psline[linewidth=0.45]{<-}(10,67)(10,9)
\newrgbcolor{userFillColour}{0.6 0.6 1}
\pscustom[linestyle=none,fillcolor=userFillColour,fillstyle=solid]{\psline(10,9)(110,60)
\psline(110,60)(140,9)
\psline(140,9)(10,9)
\psbezier(10,9)(10,9)(10,9)
}
\newrgbcolor{userLineColour}{0.6 0.6 1}
\psline[linewidth=0.5,linecolor=userLineColour](10,9)(110,60)
\psline[linewidth=0.5](140,9)(140,7)
\psline[linewidth=0.5](8,60)(10,60)
\rput(4,60){$\frac{1}{d+1}$}
\rput(110,3){$\frac{1}{d+1}$}
\rput(140,3){$\frac{1}{d}$}
\rput(152,9){$\frac{1}{q}$}
\rput(4,67){$\frac{1}{\alpha}$}
\newrgbcolor{userLineColour}{0.6 0.6 1}
\psline[linecolor=userLineColour](10,9)(110,60)
\newrgbcolor{userFillColour}{1 1 1}
\rput{0}(110,60){\psellipse[fillcolor=userFillColour,fillstyle=solid](0,0)(1,-1)}
\psline{->}(56,46)(62,37)
\rput(54,50){$\frac{1}{\alpha}=\frac{1}{q}$}
\psline[linewidth=0.45]{->}(10,9)(147,9)
\newrgbcolor{userLineColour}{0.6 0.6 1}
\psline[linecolor=userLineColour](10,9)(140,9)
\newrgbcolor{userFillColour}{0.6 0.6 1}
\rput{0}(140,9){\psellipse[fillcolor=userFillColour,fillstyle=solid](0,0)(1,-1)}
\rput(8,7){$0$}
\psline[linewidth=0.5](8,44)(10,44)
\rput(4,44){$\frac{1}{d+2}$}
\psline(135,39)(137,37)
\psline(137,39)(135,37)
\rput{0}(118,44){\psellipticarc[linecolor=red](0,0)(1,-1){-90}{90}}
\psline[linecolor=red](142,44)(119,44)
\end{pspicture}

  }
  \caption{New region of validity if \eqref{eq:schatten-strichartz-mixed-2} holds at some other point.\label{fig:argument-2}} 
  \end{figure}
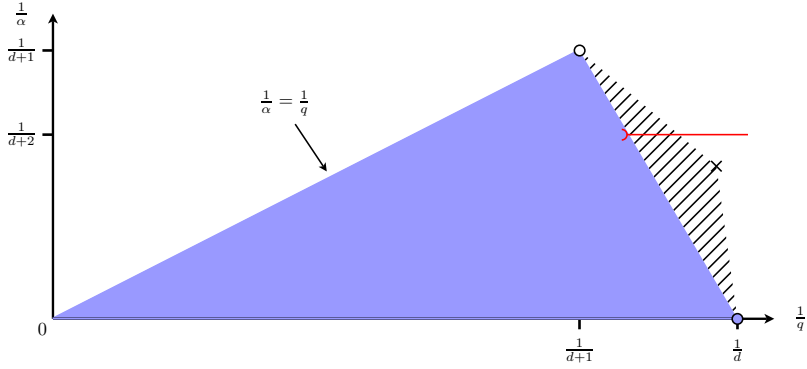
\end{center}
 
 We now prove the claim. Let $v=v(t,x)$ a bounded, non-negative, non-zero function on $\R\times\R^d$ such that 
 $$|t|\ge1/2\Longrightarrow v(t,\cdot)\equiv0,$$
 and let $N\in\N$, $T>1$. The parameter $T$ will be chosen later on to depend on $N$, and the parameter $N$ will go to infinity. Taking $W_1=\bar{W}$, $W_2=W$, $V=|W|^2$, and using a similar argument as in the proof of Lemma \ref{lem:optimality-keel-tao}, the inequality \eqref{eq:schatten-strichartz-mixed-2} is equivalent to 
 $$\norm{\int_\R e^{it\Delta}V(t,x)e^{-it\Delta}\,dt}_{\gS^{\alpha}(L^2_x(\R^d))}\lesssim\norm{V}_{L^{p/2}_tL^{q/2}_x(\R\times\R^d)}.$$
 We choose a trial $V$ to be of the form
 $$V(t,x)=\sum_{j=1}^N v(t-jT,x),$$
 such that
 $$\norm{V}_{L^{p/2}_tL^{q/2}_x}\simeq N^{2/p}\norm{v}_{L^{p/2}_tL^{q/2}_x}\simeq N^{2/p}.$$
 Define 
 \begin{align*}
    A_j:=\int_\R e^{it\Delta}v(t-jT,x)e^{-it\Delta}\,dt &= e^{ijT\Delta}\left(\int_\R e^{it\Delta}v(t,x)e^{-it\Delta}\,dt\right)e^{-ijT\Delta} \\
    &=e^{ijT\Delta}A_ve^{-ijT\Delta},
 \end{align*}
 and recall that $A_v\in\gS^{d+2}$ by the Strichartz inequality of \cite{FraLewLieSei-13}. Then,
 $$\tr\left(\int_\R e^{it\Delta}V(t,x)e^{-it\Delta}\,dt\right)^{d+2}=N\tr A_v^{d+2}+\sum_{(j_1,\ldots,j_{d+2})\notin D}\tr A_{j_1}\cdots A_{j_{d+2}},$$
 where
 $$D=\Big\{(j,\ldots,j)\in\{1,\ldots,N\}^{d+2},\,j=1,\ldots,N\Big\}.$$
 For any $(j_1,\ldots,j_{d+2})\notin D$, there exists $1\le k\le d+1$ such that $j_k\neq j_{k+1}$. By cyclicity of the trace, we may assume that $k=1$, so that 
 $$\tr A_{j_1}\cdots A_{j_{d+2}}=\tr A_v e^{i(j_2-j_1)T\Delta} A_v e^{-ij_2T\Delta} A_{j_3}\cdots A_{j_{d+2}}e^{ij_1T\Delta}.$$
 Since $j_2\neq j_1$, the operator $A_v e^{i(j_2-j_1)T\Delta} A_v$ converges strongly to zero in $\gS^{\frac{d+2}{2}}$ as $T\to+\ii$, by the following lemma.
 
 \begin{lemma}
  Let $A,B\in\gS^{d+2}$. Then, the operator $Ae^{it\Delta}B$ converges strongly to $0$ in $\gS^\frac{d+2}{2}$ as $t\to+\ii$.
 \end{lemma}
 
 \begin{proof}
  We may assume that $B$ is finite-rank by density. Thus, assume $B=\sum_j\lambda_j|u_j\rangle\langle v_j|$. We have $\norm{Ae^{it\Delta}|u_j\rangle\langle v_j|}_{\gS^p}=\norm{Ae^{it\Delta}u_j}\norm{v_j}$. Now $e^{it\Delta}u_j$ goes weakly to zero in $L^2$ as $t\to+\ii$, and this implies that $Ae^{it\Delta}u_j$ goes strongly to zero in $L^2$ since $A$ is compact. 
 \end{proof}

  Since the operator $e^{-ij_2T\Delta}A_{j_3}\cdots A_{j_{d+2}}e^{ij_1T\Delta}$ is uniformly bounded in $\gS^{\frac{d+2}{d}}$, we deduce that $A_v e^{i(j_2-j_1)T\Delta} A_v e^{-ij_2T\Delta} A_{j_3}\cdots A_{j_{d+2}}e^{ij_1T\Delta}$ converges strongly to $0$ in $\gS^1$ as $T\to+\ii$. As a consequence, there exists $T=T(v,N)>1$ large enough such that
  $$\left|\sum_{(j_1,\ldots,j_{d+2})\notin D}\tr A_{j_1}\cdots A_{j_{d+2}}\right|\le \frac{N}{2}\tr A_v^{d+2}.$$
  We thus find that $N^\frac{1}{d+2}\lesssim N^\frac{2}{p}$ and hence $2/p\ge 1/(d+2)$, which proves the proposition recalling that $2/p=1-d/q$.
\end{proof}

\subsection{An application: a refined Strichartz estimate}

As was hinted in the introduction, the Schatten bounds \eqref{eq:schatten-strichartz-mixed} are equivalent to bounds on orthonormal functions: if $u_1,\ldots,u_M$ is an orthormal system in $L^2_x(\R^d)$ and if $\nu_1,\ldots,\nu_M\in\C$ are corresponding coefficients, Theorem \ref{thm:schatten-strichartz-mixed} implies that

\begin{equation}\label{eq:orthonormal-coef}
  \norm{\sum_{j=1}^M\nu_j|e^{it\Delta}u_j|^2}_{L^{(p/2)'}_tL^{(q/2)'}_x}\le C\left(\sum_{j=1}^M|\nu_j|^{\alpha'}\right)^{1/{\alpha'}}.
\end{equation}

We refer to \cite{FraSab-14} for the explanation as to why Theorem \ref{thm:schatten-strichartz-mixed} indeed implies such a bound. One way to obtain a system of orthogonal functions from a single function $u\in L^2_x(\R^d)$ is to localize it on $M$ disjoint sets in physical or Fourier variables. For instance, denoting by $(P_j)_{j\in\Z}$ the standard Littlewood-Paley multipliers for functions on $\R^d$ (see for instance \cite[Sec. 8.2]{MusSch-book}), and fixing a function $u\in L^2_x(\R^d)$, then the functions $u_j:=P_ju$ are (almost) orthogonal and we may try to apply \eqref{eq:orthonormal-coef} to it. What is interesting in this approach is that the left side of \eqref{eq:orthonormal-coef} controls the $L^{(q/2)'}_x$-norm of $e^{it\Delta}u$ by the Littewood-Paley theorem:

$$\norm{\sum_{j\in\Z}|e^{it\Delta}P_ju|^2}_{L^{(q/2)'}_x}\gtrsim\norm{e^{it\Delta}u}_{L^{2(q/2)'}_x}^2.$$

Of course, the issue with this approach is that (i) the functions $u_j=P_ju$ are not normalized in $L^2_x$, however their $L^2$-norm can be put into the coefficients $\nu_j$; (ii) the $(u_j)$ are not exactly orthogonal, but only $u_j$, $u_{j+1}$, $u_{j-1}$ have an overlap. Hence, one may still compute the Schatten norm of the operator $\sum_j|P_ju\rangle\langle P_ju|$, see \cite[Cor. 9]{FraSab-14} for details. Putting all these remarks together, we arrive at the following result \cite[Cor. 9]{FraSab-14}:

\begin{proposition}[A refined Strichartz estimate]\label{prop:refined}
 Let $d\ge1$ and $p,q\ge2$ satisfying
 $$\frac2p+\frac dq=\frac d2,\qquad 2\le q< 2+\frac{4}{d-1}.$$
 Then, there exists $C>0$ such that for any $u\in L^2_x(\R^d)$ we have
 \begin{align}\label{eq:refined1}
    \norm{e^{it\Delta}u}_{L^p_tL^q_x(\R\times\R^d)} &\le C\left(\sum_{j\in\Z}\norm{P_ju}_{L^2_x}^{4q/(q+2)}\right)^{(q+2)/(4q)}\\
    &\le C'\left(\sup_{j\in\Z}\norm{P_ju}_{L^2_x}\right)^{(q-2)/(2q)}\norm{u}_{L^2_x}^{(q+2)/(2q)}. \label{eq:refined2}
 \end{align}
\end{proposition}

The fact that \eqref{eq:refined2} follows from \eqref{eq:refined1} is due to the fact that $\sum_j\norm{P_ju}_{L^2_x}^2$ is comparable to $\norm{u}_{L^2_x}^2$ and that $4q/(q+2)\ge2$. A similar estimate can be obtained in the range $2+4/(d-1)\le q <2+4/(d-2)$ using the other set of Schatten bounds mentioned in the previous section, obtained by interpolation with the Keel-Tao endpoint estimate.

Refined Strichartz estimates of the type \eqref{eq:refined2} appeared first in the work of Bourgain \cite{Bourgain-98} when $d=2$, and was later generalized by B\'egout-Vargas \cite{BegVar-07} for $d\ge3$ and Carles-Keraani \cite{CarKer-07} for $d=1$. It is interesting to notice that in these works, the refined estimates follow from deep bilinear estimates, for instance of Tao \cite{Tao-03}, while Proposition \ref{prop:refined} follows from a much simpler argument. On the other hand, Proposition \ref{prop:refined} does not imply the profile decomposition in the mass-critical case, because it does not detect the translation in Fourier space (Galilean boosts). 

\section{Elements of proof}\label{sec:proof}

We now turn to the proof of Theorem \ref{thm:main} and Theorem \ref{thm:schatten-strichartz-mixed}, which are both based on the same idea following the original works of Stein and Strichartz. They wanted to prove that the operator
$$T_S:L^p(\R^N)\to L^{p'}(\R^N)$$
is bounded. Their idea is to introduce a bounded, analytic family of operators $(G_z)$ depending on a complex parameter $z\in\C$ living on a strip $a\le\text{Re}\, z\le b$ for some real numbers $a<b$, which satisfies $G_c=T_S$ for some $c\in(a,b)$. The interest is to regularize the singular distribution $T_S$, which a Fourier multiplier by a distribution supported on the surface $S$, by an analytic deformation. It is practically done in the same fashion as one can regularize a delta-distribution $\delta_0$ at the origin on $\R$, by the family $x\mapsto x_+^z$ (one recovers $\delta_0$ at $z=-1$ where this distribution has a pole; however one may compensate the pole at $z=-1$ by multiplying with a well-chosen, $x$-independent function $g(z)$ like $g(z)=1/\Gamma(z)$). 

With this suitable choice of regularization, Stein and Strichartz were able to prove bounds of the type 
$$
\begin{cases}
 \norm{G_{a+is}}_{L^1\to L^\ii} \le C, \\
 \norm{G_{b+is}}_{L^2\to L^2} \le C,
 \end{cases}
$$
for all $s\in\R$, and for some $C>0$ independent of $s$. Applying Stein's interpolation theorem \cite{Stein-56} (which is a generalization of Hadamard's three-line lemma), one obtains the correct $L^p\to L^{p'}$ bound for the operator $G_c=T_S$.

We now want to derive the bounds \eqref{eq:stein-tomas-schatten-2} using similar ideas. Using the same analytic family $(G_z)$, we derive the bounds
$$
\begin{cases}
 \norm{W_1G_{a+is}W_2}_{\gS^2} \le C\norm{W_1}_{L^2}\norm{W_2}_{L^2}, \\
 \norm{W_1G_{b+is}W_2}_{L^2\to L^2} \le C\norm{W_1}_{L^\ii}\norm{W_2}_{L^\ii}.
 \end{cases}
$$
The second one follows trivially from the $L^2\to L^2$ bound on $G_{b+is}$ derived by Stein and Strichartz, while the first one is not much harder: indeed, one may express the Hilbert-Schmidt norm in $\gS^2$ in terms of the integral kernel of the considered operator:
$$\norm{W_1G_{a+is}W_2}_{\gS^2}^2=\int_{\R^N}\int_{\R^N}|W_1(x)|^2|G_{a+is}(x,y)|^2|W_2(y)|^2\,dx\,dy,$$
which together with the remark that
$$\norm{G_{a+is}}_{L^1(\R^N)\to L^\ii(\R^N)}=\norm{G_{a+is}(\cdot,\cdot)}_{L^\ii(\R^N\times\R^N)},$$
shows the desired inequality. One may then interpolate the $\gS^2$ and the $L^2\to L^2$ bounds in the same fashion as in \cite{Simon-79} to obtain $\gS^\alpha$ bounds on $W_1G_cW_2=W_1T_SW_2$ which proves Theorem \ref{thm:main} and Theorem \ref{thm:schatten-strichartz-mixed}.

As is clear from our strategy of proof, we did not use at all that we were in the context of restriction inequalities. Indeed, our strategy carries on as long as some operator $T$ is shown to be bounded from $L^p$ to $L^{p'}$ by (complex) interpolating $L^2\to L^2$ and $L^1\to L^\ii$ bounds. Our proof shows that, automatically, one may improve the $L^p\to L^{p'}$ bound into a stronger Schatten bound. In \cite{FraSab-14}, we provide other examples where we can apply the same principle, and some applications of these Schatten bounds.

\subsection*{Acknowledgments}

Lemma \ref{lem:optimality-keel-tao} and Proposition \ref{prop:new-necessary-condition} are new results and do not appear anywhere else. They have been obtained in collaboration with Mathieu Lewin.

% %%%%%%%%%%%%%%%%%%%%%%%%%%%%%%%%%%%%%%%%%%
% %%%%%%%%%%%%%%%%%%%%%%%%%%%%%%%%%%%%%%%%%%
% \bibliographystyle{siam}
% \bibliography{biblio}

\end{document}